\documentclass[12pt]{amsproc}
\usepackage[foot]{amsaddr}
\usepackage{amsmath,amssymb,amsthm}
\usepackage{color,graphics,srcltx,float,anysize,framed,nicematrix}
\usepackage{stmaryrd}
\usepackage{enumerate}
\usepackage{booktabs}
\usepackage{makecell}
\usepackage[backref]{hyperref}
\hypersetup{
    colorlinks,
    linkcolor={red!60!black},
    citecolor={green!60!black},
    urlcolor={blue!60!black}
}

\usepackage{paralist}
\usepackage{soul}

\definecolor{darkblue}{rgb}{0,0,0.8}

\newtheorem{theorem}{Theorem}[section]
\newtheorem{lemma}[theorem]{Lemma}

\newtheorem{conjecture}[theorem]{Conjecture}
\theoremstyle{definition}

 \newtheorem{open}[theorem]{Open Problem}

\DeclareMathOperator{\PG}{{\mathrm{PG}}}

\DeclareMathOperator{\PSU}{{\mathrm{PSU}}}
\DeclareMathOperator{\PGU}{{\mathrm{PGU}}}
\DeclareMathOperator{\PSp}{{\mathrm{PSp}}}

\DeclareMathOperator{\POmega}{{\mathrm{P}\Omega}}

\DeclareMathOperator{\PGammaU}{{\mathrm{P\Gamma U }}}
\DeclareMathOperator{\PGammaSp}{{\mathrm{P\Gamma Sp}}}
\DeclareMathOperator{\PGammaO}{{\mathrm{P\Gamma O}}}

\DeclareMathOperator{\tr}{{\mathrm{Tr}}}

\DeclareMathOperator{\GL}{{\mathrm{GL}}}

\DeclareMathOperator{\F}{\mathbb{F}}
\DeclareMathOperator{\Aut}{{\mathrm{Aut}}}

\renewcommand{\le}{\leqslant}
\renewcommand{\ge}{\geqslant}

\usepackage{color}

\allowdisplaybreaks

\title[Tactical decompositions in finite polar spaces and\dots]{Tactical decompositions in finite polar spaces and non-spreading classical group actions}

\author[Bamberg]{John Bamberg$^\dagger$}
\author[Giudici]{Michael Giudici$^\dagger$}
\author[Lansdown]{Jesse Lansdown$^\ddagger$}
\author[Royle]{Gordon F. Royle$^\dagger$}

\address{$^\dagger$Centre for the Mathematics of Symmetry and Computation, Department of Mathematics and Statistics, The University of Western Australia, Perth, WA 6009, Australia.}
\email{$^\dagger$firstname.lastname@uwa.edu.au}

\address{$^\ddagger$School of Mathematics and Statistics, University of Canterbury, 
Christchurch, New Zealand}
\email{$^\ddagger$firstname.lastname@canterbury.ac.nz}

\dedicatory{Dedicated to the memory of Kai-Uwe Schmidt.}

\begin{document}

\maketitle

\begin{abstract}
For finite classical groups acting naturally on the set of points of their ambient polar spaces,
the symmetry properties of \emph{synchronising} and \emph{separating} are equivalent
to natural and well-studied problems on the existence of certain configurations
in finite geometry. The more general class of \emph{spreading} permutation groups
is harder to describe, and it is the purpose of this paper to explore
this property for finite classical groups. In particular,
we show that for most finite classical groups, their natural action on the points 
of its polar space is non-spreading. We develop and use a result on tactical 
decompositions (an \emph{AB-Lemma})
that provides a useful technique for finding witnesses for non-spreading permutation groups.
We also consider some of the other primitive actions of the classical groups.
\end{abstract}

\section{Introduction}

Permutation groups can be naturally separated into classes between primitive and $2$-homogeneous according to the \emph{synchronisation hierarchy}.
The three major layers in the hierarchy are the \emph{synchronising}, \emph{separating}, and \emph{spreading} permutation groups. In general, the synchronisation and separation properties are better understood than the spreading property, since the construction (or proof of absence) of a suitable combinatorial witness is typically more tractable for these properties. A transitive permutation group $G$ acting on a set $\Omega$ is \emph{non-spreading} if there
exists a nontrivial multiset $X$ and a nontrivial set $Y$ such that $|X|$ divides $|\Omega|$ and $|X\star Y^g|$ is constant for all $g\in G$ and is \emph{spreading} otherwise\footnote{Here $\star$ denotes intersection which accounts for repetition in the multiset, more formally, $(X \star Y)(i) = X(i)Y(i)$ for multisets $X$ and $Y$, where $X(i)$ is the multiplicity of $i$ in $X$, and similarly for $Y(i)$. We will call a multiset \emph{trivial} if it is constant or there is only one element with nonzero multiplicity.}. Notice that a transitive subgroup of a non-spreading group is non-spreading.
Since such witnesses are difficult to determine, the spreading property was previously resolved with witnesses involving proper multisets in only a few ad-hoc cases. Most other examples either satisfied the stronger non-separating property or in some instances have obvious witnesses involving sets (for instance, the symmetric group $S_n$ acting on $k$-sets is always non-spreading \cite{AraujoCameronSteinberg2017}). In this paper we pioneer the first systematic approach to the generation of witnesses for non-spreading permutation groups in Theorem \ref{thm:normalsubgroup} by generalising and adapting the so-called ``AB-Lemma''\cite{Bayens} from finite geometry.
This result has also been used recently in showing that primitive groups of diagonal type
are non-spreading \cite{withSaul}.
Moreover, in Subsection \ref{sect:spreading} and 
Section \ref{sect:unitary_kspaces} we show that the classical groups in many of their natural actions are non-spreading, using a combination of Theorem \ref{thm:normalsubgroup} and suitable tactical decompositions.
In particular, we resolve the spreading question for the classical groups acting on totally singular $1$-spaces,  apart from the unitary case. We resolve the unitary case for dimension $5$ conditional upon Conjecture \ref{conjecture:tricky} about certain solutions to equations in finite fields. Finally, in Section \ref{sec:halfWitt} and the end of Section \ref{sect:unitary_kspaces}, we consider actions on higher dimensional subspaces that are also non-spreading. Since the classical groups are so intimately connected to the classical polar spaces, many of these results are described in terms of the geometric objects known as \emph{$k$-ovoids}.

\section{Tactical decompositions and the ``$AB$-Lemma''}

An incidence structure $(\mathcal{P},\mathcal{B},\mathcal{I})$ is 
a triple of sets with $\mathcal{I}\subseteq \mathcal{P}\times \mathcal{B}$. For instance,
the points and lines of the Euclidean plane, together with the incidence relation between them, form an incidence structure. It is customary to call the elements of $\mathcal{P}$ ``points'', 
and the elements of $\mathcal{B}$ ``blocks''.
Now suppose we have partitions $\{P_1,\ldots, P_s\}$ of the points and  
$\{B_1,\ldots, B_t\}$ of the blocks such that every block in $B_i$ is incident with $\lambda_{i,j}$ points of $P_j$, then we call the decomposition \emph{block-tactical}. Dually, we have \emph{point-tactical} decompositions, where every point in $P_j$ is incident with a constant number of blocks of $B_i$, dependent only on $P_j$ and $B_i$. A \emph{tactical decomposition} of $(\mathcal{P},\mathcal{B},\mathcal{I})$ is both block-tactical and point-tactical (see \cite[Section 1.1]{Dembowski}).
For example, the orbits of a subgroup of automorphisms on both points and blocks produces a tactical decomposition.

In finite geometry, block-tactical decompositions with two point-classes are often known in other terms.
For example, a \emph{$k$-ovoid} of a finite polar space is a set of points such that every maximal totally singular subspace
intersects in $k$ points. If we consider the incidence structure of points and maximal totally singular subspaces (as the \emph{blocks}),
then we have a block-tactical decomposition with two points-classes: the $k$-ovoid and its complement.
 Researchers have been particularly interested in the construction
of $k$-ovoids with cardinality half the size of the points, known as \emph{hemisystems}. One technique for both proof of existence and the construction of hemisystems
has been to consider groups of automorphisms that yield small block-tactical decompositions and then attempt to fuse orbits in a suitable manner. For instance, the Cossidente-Penttila examples of $\mathsf{H}(3,q^2)$, $q$ odd, arise from the orbits of the subfield subgroup $\POmega^-(4,q)$. Indeed the matrix of the point-tactical decomposition (which records the constants $\lambda_{i,j}$) is as follows (see \cite{CossidentePenttila}):
\[
\begin{bmatrix}
\frac{q+1}{2}&\frac{q+1}{2}&0&0\\
0&0&\frac{q+1}{2}&\frac{q+1}{2}\\
1&1&\frac{q-1}{2}&\frac{q-1}{2}
\end{bmatrix}
\]
In this example, we can take the union of the first and third point-classes (according to the first and third columns), 
and the second and fourth point-classes, and we have a matrix with $(q+1)/2$ in each entry, and thus
a tactical decomposition with two point-classes.
This example can also be constructed directly by the following result of Bayens \cite{Bayens} which extends the work of Cossidente and Penttila \cite{CossidentePenttila}.

\begin{theorem}[Bayens]
Let $\mathcal{S} = (\mathcal{P},\mathcal{B},\mathcal{I})$ be an incidence structure.
Let $A$ and $B$ be two subgroups of $\Aut(\mathcal{S})$ such that (i) $B$ is a normal subgroup of $A$, (ii) $A$ and $B$ have the same orbits on $\mathcal{P}$, (iii) each $A$-orbit on $\mathcal{B}$ splits into two $B$-orbits. Then there are $2^n$ hemisystems admitting $B$, where $n$ is the number of $A$-orbits on $\mathcal{B}$.
\end{theorem}

This result, which is essentially about ``regular'' tactical decompositions arising from a group action, is known in finite geometry
circles as the \emph{AB-Lemma}. It was originally used for the construction of hemisystems of Hermitian generalised quadrangles, and later to hemisystems of finite polar spaces (\cite{CP2016}, \cite{LansdownNiemeyer}). A generalisation to $k$-ovoids was given in \cite[Theorem 6.1.2]{LansdownPhD}. A key result of this paper, Theorem \ref{thm:normalsubgroup}, is an adaptation of Bayens' AB-Lemma. It has already been a crucial component in determining the non-spreading property of
certain primitive permutation groups \cite{withSaul}.

\begin{theorem}
\label{thm:normalsubgroup}
Let $G$ be a group acting on $\Omega$, and let $A$ and $B$ be subgroups of $G$ such that $B \triangleleft A$. Let $\omega_1, \ldots, \omega_k \in \Omega$ with $k\geqslant 2$ such that  $\omega_i^B \neq \omega_j^B$ for $i \neq j$, and $\omega_1^A = \omega_1^B \cup \omega_2^B \cup \ldots \cup \omega_k^B$. Let $X \subset \Omega$, and $\Delta = \{Y \in X^G \mid Y \cap \omega_1^A \neq \varnothing\}$.
If the orbits of $A$ and $B$ on $\Delta$ are the same, then $(X,\Omega + k\omega_1^B - \omega_1^A)$ is a witness to $G$ being non-spreading (where $k \omega_1^B$ is the multiset corresponding to $\omega_1^B$ with each entry assigned multiplicity $k$).
\end{theorem}

\begin{proof}
 For $i \in \{1, \ldots, k\}$, there exists $a_i \in A$ be such that $\omega_1 = \omega_i^{a_i}$. Let $Z \in \Delta$, then there exists $b_i \in B$ such that $Z^{a_i^{-1}b_i}=Z$ since $A$ and $B$ have the same orbits on $\Delta$. Moreover, $a_iB = Ba_i$ since $B \triangleleft A$. Hence
\[
|Z \cap \omega_1^B| = |Z \cap (\omega_i^{a_i})^B| = |Z \cap (\omega_i^B)^{a_i}| 
= |Z^{a_i^{-1}b_i} \cap \omega_i^B| = |Z \cap \omega_i^B|.
\]
Instead take $Z \in X^G \backslash \Delta$. Then $|Z \cap \omega_i^B| =0$ by definition of $\Delta$.
It follows that $|X^g \cap \omega_1^B| = |X^g \cap \omega_i^B|$ for all $g \in G$ and $i\in\{1,\ldots,k\}$, and so $k|X^g \cap \omega_1^B| = |X^g \cap \omega_1^A|$. Since $k>1$, we have that  $\Omega + k\omega_1^B - \omega_1^A$ is a non-trivial multiset such that $|X^g \cap (\Omega + k\omega_1^B - \omega_1^A)| = |X|$ and $|\Omega + k\omega_1^B - \omega_1^A|= |\Omega|$ which divides $|\Omega|$. Hence $(X, \Omega + k\omega_1^B - \omega_1^A)$ is a witness to $G$ being non-spreading.
\end{proof}

Of course, it is a non-trivial exercise to determine a suitable set $X$ in Theorem \ref{thm:normalsubgroup}, not to mention suitable subgroups $A$ and $B$. Sometimes, however, the process of determining them is guided by the problem at hand. For example, given an incidence structure $(\mathcal{P},\mathcal{B},\mathcal{I})$, if we are interested in the action of $G \le {\rm Aut}\left((\mathcal{P},\mathcal{B},\mathcal{I})\right)$ on $\mathcal{P}$, then the set of points incident with a suitable block is a likely candidate for $X$. In this case $\Delta$ is (possibly a subset of) the blocks which meet some element of a given $A$-orbit on points. In this way Theorem \ref{thm:normalsubgroup} mimics the key properties of the AB-Lemma while recognising that witnesses for spreading need only satisfy these properties locally.

\section{Classical groups in their actions on singular $1$-spaces}

In this section we survey the results on classical groups which are spreading in their action on totally singular $1$-dimensional subspaces. The natural geometric setting for the classical groups are the polar spaces. We will adopt this framework, since many of the witnesses to non-spreading are given by well studied geometric objects. The projective dimension is used, so a singular $1$-dimensional subspace is called a ``point''.

Recall that spreading sits above separating and synchronising in the synchronisation hierarchy. That is, spreading implies separating, which in turn implies synchronising \cite[Corollary 5.5, Theorem 5.14]{AraujoCameronSteinberg2017}. Hence we also consider separating and synchronising in our survey, since if a group is non-synchronising or non-separating, then it is also immediately non-spreading.
For completeness, we also now give the definitions of synchronising and separating permutation groups.

Let $G$ be a permutation group acting on $\Omega$. Given a partition $\Pi$ of $\Omega$, a \emph{transversal} $X$ is a subset of $\Omega$ which meets each part of $\Pi$ in exactly one element. We say $\Pi$ is \emph{section-regular} if there is some transversal $X$ such that $X^g$ is also a transversal for all $g \in G$. We call $G$ \emph{non-synchronising} if there exists a non-trivial section-regular partition and \emph{synchronising} otherwise. We see directly from this definition that a synchronising permutation group is primitive. If there exist subsets $X$ and $Y$ of $\Omega$ such that $|X||Y|=|\Omega|$ and $|X^g \cap Y|=1$ for all $g \in G$ then $G$ is \emph{non-separating} and \emph{separating} otherwise. Notice that if $\Pi$ is a section-regular partition for $G$ acting on $\Omega$, with transversal $A$ and part $B$, then
$|A|\cdot |B|=|\Omega|$ and $|A^g\cap B|=1$ for all $g\in G$. Therefore, a non-synchronising group action is non-separating.
As with non-spreading groups, any transitive subgroup of a non-synchronising, respectively non-separating, group is also non-synchronising, respectively non-separating.

\subsection{Background on polar spaces}

The geometric setting of this paper is in the \textit{finite classical 
polar spaces}; those geometries which come from
a finite vector space equipped with a reflexive sesquilinear form or
quadratic form. We will assume that the reader is familiar with the fundamental
theory of polar spaces and we will use projective notation for polar spaces so that 
they differ from the standard notation for their collineation groups.
For more background on polar spaces see \cite{HirschfeldThas}.
For example, we will use the notation $\mathsf{W}(d-1,q)$ to denote the symplectic polar
space coming from the vector space $\F_q^d$ equipped with a
non-degenerate alternating form. We will be using algebraic dimensions throughout this paper,
but occasionally using projective terminology. So a 1-dimensional subspace will sometimes be referred
to as a \emph{point}. Here is a summary of the notation we
will use for polar spaces. Note that the action of the collineation group on the points is always primitive.

\begin{table}[H]
\begin{tabular}{ccc}
\toprule
Polar Space&Notation&Collineation Group\\
\midrule
Symplectic&$\mathsf{W}(d-1,q)$, $d$ even&$\PGammaSp(d,q)$\\
Hermitian&$\mathsf{H}(d-1,q^2)$, $d$ odd&$\PGammaU(d,q)$\\
Hermitian&$\mathsf{H}(d-1,q^2)$, $d$ even&$\PGammaU(d,q)$\\
Orthogonal, elliptic&$\mathsf{Q}^-(d-1,q)$, $d$ even&$\PGammaO^-(d,q)$\\
Orthogonal, parabolic&$\mathsf{Q}(d-1,q)$, $d$ odd&$\PGammaO(d,q)$\\
Orthogonal, hyperbolic&$\mathsf{Q}^+(d-1,q)$, $d$ even&$\PGammaO^+(d,q)$\\
\bottomrule
\end{tabular}
\medskip
\caption{Notation for the finite polar spaces.}
\end{table}

The points and maximal totally singular subspaces of $\mathcal{P}$ form an incidence structure, and so a \emph{$k$-ovoid} of $\mathcal{P}$ is a subset of points such that every maximal totally singular subspace of $\mathcal{P}$ 
has precisely $k$ points incident with it. We refer to a $1$-ovoid simply as an \emph{ovoid}. We say that a $k$-ovoid is \emph{nontrivial} if it is nonempty and not equal to the entire set of points. Exchanging points and maximal totally singular subspaces gives us the complementary notion of a \emph{spread}: a subset of totally singular subspaces of $\mathcal{P}$ which partition the points.

\subsection{Separating}

By \cite[Theorem 6.8]{AraujoCameronSteinberg2017}, a classical group acting on singular points of the associated polar space is non-separating if and only if its polar space possesses an ovoid (that is, a $1$-ovoid).
Much of the following discussion can also be found in \cite{DeBeule_et_al}.
Thas \cite{Thas81} showed that $\mathsf{W}(2n-1,q)$ does not have an ovoid for all odd $q$ and for all $n\ge 2$. 
Also, Thas showed that $\mathsf{Q}^-(2n-1,q)$ does not have an ovoid for all $n\ge 3$,
and $\mathsf{H}(2n,q^2)$ does not have an ovoid for all $n\ge 2$. 
The generalised quadrangle $\mathsf{Q}(4,q)$ has an ovoid for all $q$; just take an elliptic quadric section.
So for all $q$, $\POmega(5,q)$ is non-separating on singular points. 
For larger dimension, it is more complicated.
Firstly, $\mathsf{Q}(2n,q)$, $n \ge 3$, $q$ even, has no ovoid, because $\mathsf{Q}(2n,q)$ is isomorphic
to $\mathsf{W}(2n-1,q)$ when $q$ is even.
For $q$ odd, it is known that ovoids of $\mathsf{Q}(2n,q)$ can only exist when $n\le 3$.
Moreover, there exists an ovoid of $\mathsf{Q}(6,q)$ whenever $q$ is a power of $3$
(e.g., the \emph{Thas-Kantor} ovoids). It is known that $\mathsf{Q}(6,q)$ does not have an ovoid when 
$q$ is prime and $q\ne 3$ (see \cite{BallGovaertsStorme}).
Now for plus type orthogonal groups. Firstly, $\mathsf{Q}^+(3,q)$ always has conic sections, which are ovoids. 
Secondly, $\mathsf{Q}^+(5,q)$ always has ovoids since $\PG(3,q)$ always has spreads (take a regular spread),
and the Klein correspondence. It is conjectured that $\mathsf{Q}^+(9,q)$ does not have ovoids, but we only have
a positive result for $q\in\{2,3\}$. 
It was shown in \cite{DeBeule_et_al2} that $\mathsf{Q}^+(2n+1,q)$ does not have an ovoid for $n>q^2$, and similarly, it was shown in \cite{Klein} that $\mathsf{H}(2n+1,q^2)$ does not have an ovoid for $n>q^3$.
Finally, $\mathsf{H}(3,q^2)$ has an ovoid for all $q$; a non-degenerate hyperplane section provides a simple example. Therefore:

\begin{theorem}[\cite{Thas81}]\leavevmode
\begin{enumerate}[(a)]
\item 
$\PSp(2n,q)$ is separating on points for all odd prime-powers $q$ and for all $n\ge 2$.
\item 
$\POmega^-(2n,q)$ is separating on singular points for all prime-powers $q$ and $n\ge 3$. 
\item 
$\PSU(2n+1,q)$ is separating on totally isotropic points for all prime-powers $q$ and $n\ge 2$.
\item For all prime-powers $q$, $\PGammaO(5,q)$ is non-separating on singular points. 
\item For all $q$ a power of $3$, $\PGammaO(7,q)$ is non-separating on singular points. 
\item For all primes $p\ne 3$, $\POmega(7,p)$ is separating on singular points. 
\item For all prime-powers $q$ and for all $n\ge 4$, $\POmega(2n+1,q)$ is separating on singular points. 
\item For all prime-powers $q$, $\POmega^+(4,q)$ and $\POmega^+(6,q)$ are non-separating. 
\item For all primes $p$, $\POmega^+(8,p)$ is non-separating. 
\item For all $q$ a power of $2$ or power of $3$, $\POmega^+(8,q)$ is non-separating. 
\item For all prime-powers $q$, $\PSU(4,q)$ is non-separating. 
\end{enumerate}
\end{theorem}

\begin{open}
Does $\mathsf{Q}(6,q)$ have an ovoid if and only if $q$ is a power of $3$?
\end{open}

The best known bounds on ovoids of
Hermitian spaces and hyperbolic quadrics are due to Blokhuis and Moorhouse
\cite{Blokhuis1995,Moorhouse1996}, but they primarily involve the characteristic of the underlying field rather than the order of the field. For instance, if $p$ is prime and $q=p^h$, then  there are no ovoids of
$\mathsf{Q}^{+}(2r+1,q)$ when $p^{r}>\binom{2r+p}{2r+1}-\binom{2r+p-2}{2r+1}$. Thus for $p=2$ or $p=3$, we
see that no ovoids exist for $r\ge 4$.

\begin{open}
For which $q$ does $\mathsf{Q}^+(7,q)$ have an ovoid?
\end{open}

The smallest unresolved case is $q=25$.

\begin{open}
Does $\mathsf{Q}^+(2n-1,q)$ have an ovoid if $n\ge 5$?
\end{open}

By De Beule and Metsch \cite{DeBeuleMetsch}, $\mathsf{H}(5,4)$ does not have an ovoid.

\begin{open}
Does $\mathsf{H}(2n-1,q)$ have an ovoid for any $n\ge 3$?
\end{open}

\subsection{Synchronising}

By \cite[Theorem 6.8]{AraujoCameronSteinberg2017} a classical group acting on the singular points of the associated polar space is non-synchronising if and only if its polar space possesses either
\begin{enumerate}[(i)]
\item an ovoid and a spread; or 
\item a partition of the points into ovoids.
\end{enumerate}

\begin{theorem}
The following groups are separating and hence synchronising:
\begin{itemize}
\item $\PSp(2n,q)$ on points, for all odd $q$ and for all $n\ge 2$;
\item $\POmega^-(2n,q)$ on singular points for all prime-powers $q$ and $n\ge 3$;
\item $\PSU(2n+1,q)$ on singular points for all prime-powers $q$ and $n\ge 2$;
\item $\POmega(7,p)$ on singular points is separating for all primes $p\ne 3$;
\item $\POmega(2n+1,q)$ on singular points for all prime-powers $q$ and for all $n\ge 4$.
\end{itemize}
\end{theorem}

\subsubsection{Orthogonal-parabolic spaces}

Suppose $q$ is a power of a prime $p$.
Firstly, an elliptic quadric hyperplane section of $\mathsf{Q}(4,q)$ is an ovoid.
Since $\mathsf{Q}(4,q)$ is the dual of $\mathsf{W}(3,q)$, and self-dual for $q$ even,
it follows that $\mathsf{Q}(4,q)$ has a spread if and only if $q$ is even. So 
$\POmega(5,q)$ is non-synchronising for all even $q$. Ball \cite{Ball04} showed that an ovoid of $\mathsf{Q}(4,q)$ meets every elliptic quadric in $1$ modulo $p$ points. Ball, Storme and Govaerts \cite{BallGovaertsStorme} showed earlier that every ovoid of $\mathsf{Q}(4,p)$ is an elliptic quadric, and so it
follows from combining both results that there is no partition of the points of $\mathsf{Q}(4,p)$ into ovoids. Therefore,

\begin{theorem}\leavevmode 
\begin{enumerate}[(i)]
    \item $\PGammaO(5,q)$ on singular points is non-synchronising for all even $q$. 
    \item $\POmega(5,p)$ on singular points is synchronising for all odd primes $p$.
\end{enumerate}
\end{theorem}

However, when $q$ is not prime, there exist ovoids of $\mathsf{Q}(4,q)$ that are not elliptic quadrics. (See Blair Williams' thesis \cite{Williams} for an excellent summary.)

\begin{open}
Does $\mathsf{Q}(4,q)$ have a partition into ovoids for $q$ odd and not prime?
\end{open}

The case $q=9$ can be easily resolved by computer. There do exist $5$ disjoint Knuth ovoids, but no bigger set of disjoint ovoids.

For $q$ even, it is known that ovoids of $\mathsf{Q}(2n,q)$ can only exist when $n=2$,
because $\mathsf{Q}(2n,q)$ is isomorphic to $\mathsf{W}(2n-1,q)$ (for $q$ even).
For $q$ odd, it is known that ovoids of $\mathsf{Q}(2n,q)$ can only exist when $n\le 3$.

\begin{open}
Does $\mathsf{Q}(6,q)$ have an ovoid if and only if $q$ is a power of $3$?
\end{open}

If $q$ is an odd power of $3$, then there will also be a spread of $\mathsf{Q}(6,q)$.

\begin{theorem}
$\PGammaO(7,3^{2h+1})$ acting on singular points is non-synchronising.
\end{theorem} 

\subsubsection{Orthogonal-hyperbolic spaces}

The spaces $\mathsf{Q}^+(4n+1,q)$ cannot have spreads, since disjoint maximal totally singular subspaces must belong to different equivalence classes. (In other words, a partial spread has size at most $2$).
However, a partition of $\mathsf{Q}^+(5,q)$ into ovoids corresponds (via the Klein correspondence) to a \emph{packing} of $\PG(3,q)$,
and they exist for all $q$ (due to a result of Denniston \cite{Denniston}).

\begin{theorem}
$\POmega^+(6,q)$ on singular points is non-synchronising for all $q$.
\end{theorem}

The triality quadric $\mathsf{Q}^+(7,q)$ is interesting here because if it has an ovoid, then it also has a spread.
So 
\begin{theorem}
Separating and synchronising are equivalent for $\mathrm{P\Omega}^+(8,q)$ acting on singular points.
\end{theorem}

We strongly believe that $\POmega^+(2n,q)$, acting  on singular points, is separating for all $n\ge 5$ and for all prime powers $q$.
Failing a proof of this, can we at least show the following?

\begin{open}
Can we show that $\mathsf{Q}^+(2n,q)$ does not have a partition into ovoids for $n\ge 5$?
\end{open}

\subsubsection{Hermitian spaces, odd dimension}

By a result of Segre \cite{Segre65}, $\mathsf{H}(3,q^2)$ does not have a spread for all prime-powers $q$. Moreover, Thas \cite{Thas81} showed that $\mathsf{H}(2n-1,q^2)$ does not have a spread for all prime-powers $q$ and for all $n\ge 2$. For all prime-powers $q$, we know that $\mathsf{H}(3,q^2)$ has a \emph{fan} (see \cite{BrouwerWilbrink}); a partition
into ovoids. Therefore,

\begin{theorem}
$\PGammaU(4,q)$acting  on singular points is non-synchronising for all prime-powers $q$ and $n\ge 2$.
\end{theorem}

We strongly believe that $\mathrm{PSU}(2n,q)$, acting on singular points, is separating for all $n\ge 3$ and for all prime powers $q$.
Failing a proof of this, can we at least show the following?

\begin{open}
Can we show that $\mathsf{H}(2n-1,q^2)$ does not have a partition into ovoids for $n\ge 3$?
\end{open}

\subsection{Spreading}\label{sect:spreading}

Let $G$ be a classical group acting naturally on the set $\Omega$ of points of its associated polar space $\mathcal{P}$.
 We have the following `$k$-ovoid analogue' of the `$1$-ovoid' characterisation for non-separating groups of Cameron and Kazanidis \cite{CAMERON_KAZANIDIS_2008} (cf. \cite[Theorem 6.8]{AraujoCameronSteinberg2017}).

\begin{lemma}\label{movoidnonspreading}
If there exists a nontrivial $k$-ovoid of $\mathcal{P}$, then the full semisimilarity group $G$ of $\mathcal{P}$ is non-spreading on totally singular points.
\end{lemma}

\begin{proof}
Let $X$ be the set of points incident with a maximal totally singular subspace of $\mathcal{P}$. Then $X$ is a set, and hence also a multiset, and moreover $|X|$ divides the number of points of $\mathcal{P}$.
Let $Y$ be a nontrivial $k$-ovoid of $\mathcal{P}$, then
$|X\star Y^g|=k$ for all $g\in G$.
\end{proof}

\begin{theorem}\label{ClassicalGroupsOnTIPoints}
The polar spaces for the groups $\PGammaU(2r,q)$, $\PGammaO^+(2r,q)$, $\PGammaO(2r+1,q)$, $\PGammaSp(2r,q)$ and $\PGammaO^-(2r-2,q)$ each have a nontrivial $k$-ovoid, and so are non-spreading in their 
actions on totally singular points, for $r \geqslant 2$.
\end{theorem}

\begin{proof}
By \cite[Theorem 6.10]{AraujoCameronSteinberg2017}, the groups $\PGammaU(2r,q)$,
$\PGammaO^+(2r,q)$, and $\PGammaO(2r+1,q)$ are non-spreading in their 
actions on totally singular points.

By \cite[Corollary 3.3]{kelly2007}, the points underlying the image under (a particular) field-reduction of the
points of $\mathsf{H}(2r,q^{2})$, is a $k$-ovoid of $\mathsf{W}(2(2r+1)-1,q)$, for some $k$.
With a different choice of field-reduction, we also obtain a $k$-ovoid of $\mathsf{Q}^-(2(2r+1)-1,q)$.
According to the comment succeeding \cite[Corollary 3.4]{kelly2007}, in both cases we will have
a nontrivial $k$-ovoid. Therefore, by Lemma \ref{movoidnonspreading}, the groups $\PGammaSp(2(2r+1),q)$
and $\PGammaO^-(2(2r+1),q)$ are non-spreading in their actions on totally singular points.

Let us consider the remaining cases: $\PGammaSp(4r,q)$ and $\PGammaO^-(4r,q)$. In the former case,
we can assume $q$ is odd, since for $q$ even, $\PGammaSp(4r,q)$ acting on its symplectic space is equivalent to $\PGammaO(4r+1,q)$ acting on its quadric.
By \cite[Corollary 3.4(2)]{kelly2007}, if $\mathsf{W}(3,q^e)$ has a $k$-ovoid,
then $\mathsf{W}(4e-1,q)$ has an $k\frac{q^e-1}{q-1}$-ovoid. Now $\mathsf{W}(3,q)$ always has a $2$-ovoid, by \cite[Theorem 5.1]{BambergLawPenttila}.
Therefore, $\mathsf{W}(4e-1,q)$ has a $2\frac{q^e-1}{q-1}$-ovoid, which is nontrivial, and hence $\PGammaSp(4r,q)$
is non-spreading in its action on totally singular points.
Finally, $\mathsf{Q}^-(3,q^e)$ is an ovoid of itself, and so by \cite[Corollary 3.4(4)]{kelly2007},
$\mathsf{Q}^-(4e-1,q)$ has a $\frac{q^e-1}{q-1}$-ovoid, and it is nontrivial. Therefore, $\PGammaO^-(4r,q)$ is non-spreading in its action 
on totally singular points.
\end{proof}

\section{Unitary groups acting on totally isotropic $k$-spaces}\label{sect:unitary_kspaces}

We saw in Theorem \ref{ClassicalGroupsOnTIPoints} that most of the classical groups in the natural actions on totally singular 1-spaces,
are non-spreading, because the associated finite polar space has a (nontrivial) $k$-ovoid, for some $k$. To date,
we do not know of any nontrivial $k$-ovoids of the polar spaces $\mathsf{H}(2r,q^2)$, and so we must employ a different argument to
show that unitary groups (in odd dimension) are non-spreading.
The groups $\PSU(3,q)$ acting on totally isotropic 1-spaces, are spreading, because they are $2$-transitive. So the first case to consider is dimension 5. Our construction will depend on the following conjecture:

\begin{conjecture}\label{conjecture:tricky}
Let $q$ be an odd prime power, and let $b\in\F_{q^2}$ such that
$b^{q+1}=-1$. Let $(s,u,w)\in\F_{q^2}^3$ such that $s^{q+1}=u^{q+1}+w^{q+1}$.
For each $\kappa\in\F_q^*$, let $n(\kappa)$ be the number of values of $\lambda\in\F_{q^2}\cup\{\infty\}$ such that 
\[(w \lambda+1)(b+u \lambda)^q-(w \lambda+1)^q (b+u \lambda)=\kappa \left(b^2+1+\lambda^2 \left(s^2-u^2-w^2\right)\right)^{(q+1)/2}.\]
Then $n(\kappa)=n(-\kappa)$ for all $\kappa\in\F_q^*$.
\end{conjecture}

\begin{theorem}\label{unitary_1spaces}
Let $q$ be an odd prime power and suppose that Conjecture \ref{conjecture:tricky} holds. Then
$\PGammaU(5,q)$ in its action on the set of totally isotropic 1-spaces,
is non-spreading. 
\end{theorem}

\begin{proof} 
Consider the Hermitian form $U(x,y):=xJ\bar{y}^\top$ where $J$ is the matrix
\[
\begin{bmatrix}
    0&0&0&0&1\\
    0&0&0&1&0\\
    0&0&1&0&0\\
    0&1&0&0&0\\
    1&0&0&0&0
\end{bmatrix}
\]
and we will refer to $\mathsf{H}(4,q^2)$ as the Hermitian polar space defined by this form.
Consider the totally isotropic points of $\mathsf{H}(4,q^2)$ with coordinates lying in $\F_q$: they
form a quadric $\mathcal{Q}$. Next, consider the line $\ell$ spanned by the first two canonical basis vectors $e_1$
and $e_2$.  Then $\ell$ is a totally isotropic line intersecting $\mathcal{Q}$ in $q+1$ points. 
Let $U$ be the set of points $P$ of $\mathsf{H}(4,q^2)$ that do not lie in $\ell^\perp$ and such that: (i)
$P^\perp\cap \ell$ does not lie in $\mathcal{Q}$; (ii) the elements of $\mathcal{Q}$ lying in $P^\perp$ span a plane (i.e., not a line). 
Consider the stabiliser of $\mathcal{Q}$ and $\ell$ in $\PGU(5,q)$. Then its elements are of the form: 
\NiceMatrixOptions{ columns-width =18pt}
\[
\begin{bNiceMatrix}
\Block{2-2}<\normalsize>{M}&&0&0&0\\
&&0&0&0\\
*&*&1&0&0\\
*&*&*& \Block{2-2}<\normalsize>{J_2M^{-\top}J_2}\\
*&*&*&&
\end{bNiceMatrix}, 
\]
with all entries in $\F_q$, and with $M\in \GL(2,q)$, $J_2=\begin{bmatrix}0&1\\1&0\end{bmatrix}$. 
The asterisks denote elements of $\F_q$ satisfying certain relations, but we will not need to know
what they are precisely.

We will now completely describe the points of $U$ in coordinate-form. Let $X:=(a,b,c,d,e)$ and suppose $X\in U$. Since $X$ does not lie in $\ell^\perp$, we have $(d,e)\ne (0,0)$. Moreover, $X^\perp\cap \ell=(\bar{d},-\bar{e},0,0,0)$, and since $X^\perp\cap \ell\notin\mathcal{Q}$,
we have $d,e\ne 0$ and $d/e\notin \F_q$. We also know that $X$ is totally isotropic,
and so $a\bar{e}+\bar{a}e+b\bar{d}+\bar{b}d+c^{q+1}=0$.
Next, we want to find all the points $P$ of $\mathcal{Q}$ lying in $X^\perp$.
Let $P=(p_1,p_2,p_3,p_4,p_5)$, and note that the $p_i$ lie in $\F_q$.
Since $P$ is totally isotropic and lies in $X^\perp$, we have
$2p_1p_5+2p_2p_4+p_3^2=0$ and $p_1e+p_2d+p_3c+p_4b+p_5a=0$.
Since $e$ is nonzero, we can assume without loss of generality that $e=1$. Then
\begin{align}
    p_1&=-(p_2d+p_3c+p_4b+p_5a)\\
    a+\bar{a}&=-\bar{b}d-b\bar{d}-c^{q+1}
\end{align}
    and hence, substituting into the previous equations gives a quadratic in $(p_2,p_3,p_4,p_5)$:
    \begin{align}
2 p_2 p_4+p_3^2-2dp_2p_5-2cp_3p_5-2bp_4p_5-2ap_5^2=0
    \end{align}
In order for this quadratic form to be nondegenerate, taking the determinant of its Gram matrix
shows that $2(a+bd)+c^2\ne 0$. Therefore, the elements of $U$ are 
precisely
\[
\{(a,b,c,d,1):  a+\bar{a}+b\bar{d}+\bar{b}d+c^{q+1}=0, d\notin \F_q, 2(a+bd)+c^2\ne 0\}.
\]

Let $\sigma$ be the map that takes the $q$-th power
of each coordinate of the given point. Then $XX^\sigma$ is a line, and it intersects $\mathcal{Q}$ in
0 or 2 points: elliptic or hyperbolic. Let $X\in U$ and write $X=(a,b,c,d,1)$. Then $X^\sigma=(a^q,b^q,c^q,d^q,1)$ and the point $(a+\lambda a^q,b+\lambda b^q,c+\lambda c^q,d+\lambda d^q, 1+\lambda)$,
where $\lambda \in \F_q^*$, lies in $\mathcal{Q}$ if $\lambda^{q+1}=1$ and
the point is totally isotropic. This means, after a tedious calculation, that
$( 2(a+bd)+c^{2})^q\lambda^2+2(a+bd)+c^2=0$
which can be solved for $\lambda^2$ because $2(a+bd)+c^{2}\ne 0$ (from the condition on $X$):
\[
\lambda^2=-( 2(a+bd)+c^{2})^{1-q}.
\]
So there are 2 or 0 solutions to the quadratic depending on whether $\lambda^{q+1}=1$ is satisfied, or not.
Thus, we apply the $(q+1)/2$-th power to each side and we find the following condition:
\[
( 2(a+bd)+c^{2})^{(1-q^2)/2}=(-1)^{(q+1)/2}
\]
or in other words, $2(a+bd)+c^{2}$ is square or not.
Indeed, $U$ splits into two parts under the stabiliser of $\mathcal{Q}$ and $\ell$, depending on whether the value of $ 2(a+bd)+c^{2}$ is square or not.
Take the part $U_\boxtimes$ for which $2(a+bd)+c^{2}$ is nonsquare. 
Let $B$ be the subgroup of the stabiliser of $\mathcal{Q}$ and $\ell$ in $\PGU(5,q)$, with $\det(M)=1$.
The set $U_\boxtimes$ splits further under the action of $B$. 
For each point $(a,b,c,d,1)$ of $U_\boxtimes$, let 
\[
\delta(a,b,c,d,1):=\frac{d^q-d}{(2(a+bd)+c^{2})^{(q+1)/2}}.
\]
It is not difficult to see by direct calculation that $\delta$ is $B$-invariant.
Now a nonzero element $x\in\F_{q^2}$ is a nonsquare if and only if $x^{(q+1)/2}+(x^q)^{(q+1)/2}=0$,
that is, $(x^q)^{(q+1)/2}=-x^{(q+1)/2}$.
So 
    \begin{align*}
\delta(a,b,c,d,1)^q&=\frac{d-d^q}{\left( ( 2(a+bd)+c^{2})^q\right)^{(q+1)/2}}\\
&=-\frac{d^q-d}{-\left(  2(a+bd)+c^{2}\right)^{(q+1)/2}}\\
&=\delta(a,b,c,d,1).
    \end{align*}
Therefore, $\delta$ has its values in $\F_q^*$. 
Moreover, two points of $U_\boxtimes$ lie in a common $B$-orbit if they have the same value under $\delta$. Therefore,
the $q-1$ orbits of $B$ on $U_\boxtimes$ are parameterised by $\delta$.

Consider the involutory collineation $\tau$ with matrix $\mathrm{diag}(1,-1,1,-1,1)$.
Then $\tau$ normalises $B$, so let $A=\langle B,\tau\rangle$. Note that $A$ pairs up the $B$-orbits on $U_\boxtimes$ because
$(a,b,c,d,1)^\tau=(a,-b,c,-d,1)$ and 
\begin{align*}
\delta(a,-b,c,-d,1)&=
\frac{(-d)^q-(-d)}{(2(a+(-b)(-d))+c^{2})^{(q+1)/2}}\\
&=-\frac{d^q-d}{(2(a+bd)+c^{2})^{(q+1)/2}}\\
&=-\delta(a,b,c,d,1).
\end{align*}

Let $m$ be a line that intersects $U_\boxtimes$, and let $\omega\in U_\boxtimes$. 
Now
\[
|m\cap (\omega^\tau)^B|=|m\cap \left(\omega^B\right)^\tau|=|m^\tau\cap \omega^B|.
\]
We want to show that $|m\cap \omega^B|=|m^\tau \cap \omega^B|$.
So we are adapting the argument used in the proof
of Theorem \ref{thm:normalsubgroup}. If $\Omega$ is the set of totally isotropic points, and $X$ is a totally isotropic line
intersecting $U_\boxtimes$, then it will follow that 
$(X,\Omega + 2\omega^B - \omega^A)$ is a witness to $G$ being non-spreading, for any $\omega\in U_\boxtimes$.

By using symmetry, we can choose $\omega$ and a point of 
$\omega^B$ that $m$ intersects in. Indeed, we may suppose without loss of generality
that $m$ intersects $\omega^B$ in a point of the form 
$(1,b,0,b,1)$, and the conditions that it be totally isotropic and lie in $U_\boxtimes$ translate to $b^{q+1}=-1$, $b\notin \F_q$, and $2(b^2+1)$ is a nonsquare.
Consider the hyperplane $\Pi:X_2 + X_4=0$. 
If $(1,b,0,b,1)$ were to lie in $\Pi$, then $2b=0$, a contradiction.
So $m$ does not lie in $\Pi$ and
intersects it in a totally isotropic point of the form 
$(-w,-u,t,u,w)$ such that $t^{q+1}=2\left(u^{q+1}+w^{q+1}\right)$.
Then the elements of $U_\boxtimes$ incident with $m$ are of the form
\[
v_\lambda:=\left(1-\lambda w,b-\lambda u, \lambda t,b+\lambda u,1+\lambda w\right),\, \lambda\in\F_{q^2}\cup\{\infty\},
\]
such that $1+\lambda w\ne 0$, $(b+u\lambda)/(1+\lambda w) \notin \F_q$,
and $2 (b^2+1)+\lambda^2 \left(t^2-2 u^2-2 w^2\right)$ is a nonsquare.
Let $n(\kappa)$ be the number of values of $\lambda$ such that $\delta(v_\lambda)=\kappa$.
By Conjecture \ref{conjecture:tricky}, and setting $s:=\sqrt{2}t$, we have $n(\kappa)=n(-\kappa)$ for all $\kappa\in\F_q^*$.
It follows that $|m\cap \omega^B|=|m^\tau \cap \omega^B|$ and hence 
$(X,\Omega + 2\omega^B - \omega^A)$ is a witness to $G$ being non-spreading.
\end{proof}

Now we look at a generalisation: unitary groups acting on $k$-dimensional totally isotropic subspaces,
where $k>1$.
Let $n\ge 2$, let $d:=2n+1$, and let $q$ be a prime power. Consider the group $G:=\PGammaU(2n+1,q)$ acting on the set of maximal totally isotropic subspaces (of dimension $n$). We will use the following finite field model for the Hermitian polar space associated to $G$. Let $\tr_{q^{2d}\rightarrow q^2}$ denote the relative trace map from $\F_{q^{2d}}$ to $\F_{q^2}$. Define the following sesquilinear form on $\F_{q^{2d}}$ thought of as a vector space over $\F_{q^2}$:
\[
f(x,y):=\tr_{q^{2d}\rightarrow q^2}\left(xy^{q^d}\right).
\]
One can easily check that $f$ is Hermitian by applying the field automorphism $\sigma:x\mapsto x^{q^d}$. Before we proceed further, let
us make a few remarks about $\sigma$. 

The set of fixed elements of $\F_{q^{2d}}$ under $\sigma$ is precisely the subfield $\F_{q^d}$. If we restrict $f$ to this subfield, we see that $f(x,y)=\tr_{q^{2d}\rightarrow q^2}(xy)=\tr_{q^{d}\rightarrow q}(xy)$, since $d$ is odd. This yields a symmetric bilinear form on $\F_{q^d}$. The totally isotropic
subspaces for this form produce a \emph{subfield geometry} for the Hermitian polar space. In fact, for $q$ odd, this is a geometric embedding of the parabolic quadric $\mathsf{Q}(2n,q)$ into the Hermitian polar space $\mathsf{H}(2n,q^2)$.
For $q$ even, we have that for all $x\in\F_{q^d}$,
\[
f(x,x)=\tr_{q^{d}\rightarrow q}(x^2)=\tr_{q^{d}\rightarrow q}(x)
\]
and so the totally isotropic points for $f$, in the subfield geometry, lie in the hyperplane $\tr_{q^{d}\rightarrow q}(x)=0$. Therefore, the subfield geometry here is a symplectic space embedded in a hyperplane of the ambient space.

\begin{theorem}\label{unitary_kspaces}
$\PGammaU(2n+1,q)$ acting on the set of $k$-dimensional ($k>1$) totally isotropic subspaces is non-spreading.
\end{theorem}

\begin{proof}
Let $d=2n+1$, let $z$ be a primitive element of $\F_{q^{2d}}$, and let $\omega:=z^{(q^d-1)(q+1)}$. 
Then $\omega$ has order $\frac{q^d+1}{q+1}$ and
is an isometry for $f$ since
\[
f(x\omega,y\omega)=\tr_{q^{2d}\rightarrow q^2}\left(xy^{q^d}\omega^{q^d+1}\right)=f(x,y)
\]
(as $\omega^{q^d+1}=1$). Let $B$ be the group generated by $\omega$, in its action by right multiplication on $\F_{q^{2d}}$,
and let $A:=\langle \omega, \sigma\rangle$. Then $B$ has index 2 in $A$ since $\sigma$ normalises $B$. 
Let $x$ be a nonzero element of $\F_{q^{2d}}$. Then for all $i$, we have
\[
x^\sigma \omega^i=x^{q^d}\omega^i=x\left(x^{q^d-1}\omega^i\right).
\]
Let $\alpha:=x^{q^d-1}\omega^i$. Now it is not difficult to see that $\langle \omega\rangle$ consists of all the elements $\beta$ of $\F_{q^{2d}}$ such that $\beta^{\frac{q^d+1}{q+1}}=1$. On the other hand,
\[
\alpha^{\frac{q^d+1}{q+1}}=\left(x^{q^d-1}\omega^i\right)^{\frac{q^d+1}{q+1}}=x^{(q^d-1)\frac{q^d+1}{q+1}}=1
\]
and so $\alpha\in \langle\omega\rangle$. So we have shown that $x^\sigma \omega^i=x\alpha$ for some $\alpha\in B$, and hence $x^A\subseteq x^B$.
So $x^A=x^B$, and hence the $A$-orbits and $B$-orbits coincide on totally isotropic points.

Next, let us look at the action on maximal totally isotropic subspaces. Let $\tau\in A$ and suppose $\tau$ fuses two $B$-orbits on maximals.
Then $\tau$ is an involution and so $\tau\in B\sigma$ (because all involutions of $A$ lie in $B\sigma$). Suppose, by way of contradiction,
that there is no $\tau\in A$ fusing two $B$-orbits on maximals. That is, every element of $B\sigma$ acts trivially on $B$-orbits. 
Since $B$ acts trivially on $B$-orbits, this is equivalent to $\sigma$ acting trivially on $B$-orbits (on maximals). 
Now suppose $\sigma$ fixes $m^B$ setwise (where $m$ is a maximal). Since the orbits of $\langle \sigma\rangle$ on $m^B$ have size $1$ or $2$, there must be a fixed element $m_0$ of $m^B$,
because $|m^B|=|B|=\frac{q^d+1}{q+1}$ and $\frac{q^d+1}{q+1}$ is odd.
This means that $m_0$ lies in the fixed substructure of $\sigma$, which is a parabolic quadric for $q$ odd, and a symplectic space for $q$ even. Geometrically, this means that every element of $m^B$ is fixed by $\sigma$. This is a contradiction, since the maximals fixed by $\sigma$
are geometrically distinguished by the subfield geometry; or in other words, we knew a priori that there are maximals not fixed by $\sigma$.

Therefore, there is a $B$-orbit $M$ of maximal totally singular subspaces that is not an $A$-orbit on maximals. The result for maximals follows from Lemma \ref{thm:normalsubgroup}, where $\Delta$ consists of the pencils of maximals on points (a pencil is the set of maximals which contain the specified point).

We can extend this construction easily to totally isotropic $k$-spaces. The above application
of Lemma \ref{thm:normalsubgroup} produces a witness $(X,\Omega + 2\omega_1^B - \omega_1^A)$ 
where $\Omega$ is the set of maximals, $\omega_1\in M$, and $X$ is the set of maximals on a point $P$, say.
Now let $\Omega_k$ be the set of singular $k$-spaces, let $X_k$ be the set of $k$-spaces
incident with $P$, and let $K$ be the subset of $k$-spaces which are contained in $\omega_1$. Now let $K_B = \sum K^B$ and $K_A = \sum K^A$ be multisets. A witness is then given by $(X_k, \Omega_k+2K_B-K_A)$ 
for the action of $G$ on totally singular $k$-spaces to be non-spreading.
\end{proof}

\section{Classical groups on $r/2$-spaces}\label{sec:halfWitt}

In general, with a few caveats, the action of classical groups on the set of all totally singular spaces of a fixed dimension are primitive (cf. \cite{King}). These actions are not well studied with respect to the synchronisation hierarchy, and we have mainly been considering the most natural of these actions; on totally singular points. However, for the remainder of this section we will also look at a particular case of \emph{half-Witt-index} subspaces.
According to \cite[p.56]{AraujoCameronSteinberg2017},
Spiga showed by computer that $\PSp(4,3)$, $\PSp(4,5)$, and $\PSp(4,7)$ acting on points are non-spreading. This result, and the more general case for $\PSp(2r,q)$ acting on totally singular points, follow immediately
from our knowledge of $k$-ovoids of symplectic spaces (see Theorem \ref{ClassicalGroupsOnTIPoints}).  Here we
provide another approach that demonstrates the direct use of the definition
of a non-spreading group, and in doing so we are able to answer the spreading question for additional permutation representations, namely the action on totally singular $r/2$-spaces.

In what follows, let $V$ be a $2r$-dimensional formed vector space over $\mathbb{F}_q$ with Witt index $r$, for $r>1$ and even. Given a subspace $U$ of $V$, let $[U]$ denote the set of totally singular $\frac{r}{2}$-dimensional subspaces of $V$ contained in $U$. (We remind the reader we are using algebraic dimensions here.)

\begin{lemma}\label{lem:intersection}
Let $S$ be a non-degenerate $r$-dimensional subspace of $V$, with the same type as $V$ (so it has Witt index $\frac{r}{2})$. For all maximal totally isotropic subspaces, $M$, either
\[
\bigg|[M] \cap [S]\bigg| = \bigg|[M] \cap [S^\perp]\bigg| = 1,
\]
or 
\[
\bigg|[M] \cap [S]\bigg| = \bigg|[M] \cap [S^\perp]\bigg| = 0.
\]
\end{lemma}

\begin{proof}
Note the intersection of $M$ with $S$ and $S^\perp$ is totally isotropic, and since the Witt index of $S$ and $S^\perp$ is $\frac{r}{2}$, it follows that at most one $\frac{r}{2}$-dimensional totally isotropic subspace of $S$ or $S^\perp$ can be contained in $M$. 
Assume that $P$ is a totally isotropic $\frac{r}{2}$-dimensional subspace of $S$ which is also contained in $M$. Note that $V = S \oplus S^\perp$, and so $P$ and $S^\perp$ intersect trivially. Both $M$ and $S^\perp$ are contained in $P^\perp$, and hence they must intersect in a subspace of dimension at least $\frac{r}{2}$, since they each have dimension $r$ and $\dim(P^\perp)= \frac{3r}{2}$. Since $M$ contains $P$, an $\frac{r}{2}$-dimensional subspace, and $S^\perp$ does not, it follows that $\dim(M \cap S^\perp)=\frac{r}{2}$. Moreover, this subspace is totally isotropic, since $M$ is totally isotropic. 
\end{proof}

The following result shows that, for $r$ even, the permutation groups arising from the actions of $\PGammaSp(2r, q)$, $\PGammaU(2r, q)$, and $\PGammaO^+(2r, q)$ on totally singular $\frac{r}{2}$-spaces are non-spreading.

\begin{theorem}
Let $G$ be the full semisimilarity group of a formed space $V$ of even Witt index $r$, acting on the set of all $\frac{r}{2}$-dimensional totally isotropic subspaces. Then $G$ is non-spreading.
\end{theorem}

\begin{proof}
Let $\Omega$ be the set of totally singular $\frac{r}{2}$-dimensional subspaces.
Let $M$ be any maximally totally isotropic subspace of $V$, and let $S$ be a non-degenerate subspace of $V$ of dimension $r$. Then clearly it follows that
\[
\left|[M] \cap [S] - [M] \cap [S^\perp]\right| = 0
\]
 from Lemma \ref{lem:intersection}, and hence
\[
\left|[M] \star \left( \Omega + [S] - [S^\perp]\right)\big| = \big| [M] \right|.
\]
Moreover, 
\[
\left|\Omega \star \bigg( \Omega + [S] - [S^\perp]\bigg) \right| = |\Omega|,
\]
which divides $|\Omega|$.
Since this is true for all $M$, it is true for $M^g$ for all $g \in G$, and so the set $Y = [M]$ and the multiset $X = \Omega + [S] - [S^\perp]$ form witnesses that $G$ is non-spreading.
\end{proof}

\section{Concluding remarks on actions on other subspaces}

The authors of \cite{Dh1} define a \emph{generalized ovoid} as a set of totally isotropic subspaces of a finite classical polar space such that each maximal totally isotropic subspace contains precisely one member of that set.
A generalized ovoid is \emph{homogeneous} if all its elements have the same dimension, $d$ say. If
a (nonempty proper) homogeneous generalized ovoid exists for a polar space $\mathcal{P}$,
then the collineation group of $\mathcal{P}$ is non-separating and hence also non-spreading in its action on totally isotropic $d$-spaces,
and the proof is very similar to the proof of Lemma \ref{movoidnonspreading}.
The conclusion to \cite{Dh1} mentions the possibility of studying more general sets of totally isotropic subspaces, where each maximal totally isotropic subspace contains a constant number of elements of the set.
These are called \emph{regular $m$-ovoids} in \cite{Dh2}, and they
also provide witnesses to non-spreading for the action of a collineation
group of a finite polar space on its totally isotropic $k$-spaces.
More generally, certain \emph{designs} in finite polar spaces would also be witnesses 
to non-spreading. This has been explored
by Weiss in \cite{Weiss}. Indeed, the main result of \cite{Weiss} shows that
for large enough parameters, the action of the collineation group of
a finite polar space on totally isotropic $k$-spaces is non-spreading.

\subsection*{Acknowledgements} This work forms part of an Australian Research Council Discovery Project DP200101951.

\bibliographystyle{plain}

\end{document}